\documentclass[11pt,leqno,a4paper]{amsart}

\usepackage[margin=3cm,top=4cm,bottom=4cm,footskip=1cm,headsep=1.5cm]{geometry}
\usepackage{amsmath,amssymb,amsthm,amsaddr,amsthm}

\usepackage{enumitem}
\usepackage[utf8]{inputenc}
\usepackage[british]{babel}

\usepackage{graphicx}
\usepackage{tikz}
\usepackage{subcaption}
\usepackage{array}

\title[Convection--diffusion problems on networks]{On the transport limit of singularly perturbed convection-diffusion problems on networks}
\author{H. Egger \and N. Philippi }
\address{Department of Mathematics, TU Darmstadt, Germany}
\email{egger@mathematik.tu-darmstadt.de}
\email{philippi@mathematik.tu-darmstadt.de}

\newtheorem{theorem}{Theorem}

\newtheorem{lemma}[theorem]{Lemma}

\theoremstyle{definition}
\newtheorem{remark}[theorem]{Remark}
\newtheorem{assumption}[theorem]{Assumption}

\def\eps{\epsilon}

\def\dx{\partial_x}
\def\dxx{\partial_{xx}}
\def\dt{\partial_t}

\def\ddt{\frac{d}{dt}}

\def\E{\mathcal{E}}
\def\V{\mathcal{V}}
\def\X{\mathcal{X}}
\def\D{\mathcal{D}}
\def\A{\mathcal{A}}

\begin{document}

\begin{abstract}
We consider singularly perturbed convection-diffusion equations on one-dimensional networks (metric graphs)  as well as the transport problems arising in the vanishing diffusion limit. Suitable coupling condition at inner vertices are derived that guarantee conservation of mass as well as dissipation of a mathematical energy which allows us to prove stability and well-posedness. 
For single intervals and appropriately specified initial conditions, it is well-known that the solutions of the convection-diffusion problem converge to that of the transport problem with order $O(\sqrt\eps)$ in the $L^\infty(L^2)$-norm with diffusion $\eps \to 0$.
In this paper, we prove a corresponding result for problems on one-dimensional networks. The main difficulty in the analysis is that the number and type of coupling conditions changes in the singular limit which gives rise to additional boundary layers at the interior vertices of the network.  
Since the values of the solution at these network junctions are not known a-priori, the asymptotic analysis requires a delicate choice of boundary layer functions that allows to handle these interior layers.
\end{abstract}

\maketitle

\vspace*{-1em}

\begin{quote}
\noindent 
{\small {\bf Keywords:} 
convection-diffusion problems, singular perturbations, partial differential equations on networks}
\end{quote}

\begin{quote}
\noindent
{\small {\bf AMS-classification (2000):}
35B25,
35K20,
35R02,
76M45
}
\end{quote}

\section{Introduction} \label{sec:intro}

The transport and diffusion of a chemical substance in the stationary flow of an incompressible fluid through a pipe can be described by 
\begin{align} \label{eq:1}
  a \dt u_\eps(x,t) + b \dx u_\eps(x,t) &= \eps \dxx u_\eps(x,t), 
\end{align}
which is assumed to hold for $x \in (0,\ell)$ and $t>0$. Here, $u$ is the concentration of the substance, $a$, $\ell$ are the cross-section and length of the pipe, $b$ is the constant flow rate, and $\eps>0$ is the diffusion coefficient. The system is complemented by boundary conditions 
\begin{align} 
u_\eps(0,t) &= \hat u_\eps^0(t) \label{eq:2}
\qquad \text{and} \qquad 
u_\eps(\ell,t) = \hat u_\eps^\ell(t), 
\end{align}
and by specifying $u_\eps(x,0)$ at initial time $t=0$. 
In the vanishing diffusion limit $\eps \to 0$, the flow of the substance in the fluid is characterized by the simpler transport equation
\begin{align} \label{eq:4}
a \dt u(x,t) + b \dx u(x,t) &= 0.
\end{align}
Assuming $b>0$, this system is to be complemented by an inflow boundary condition  
\begin{align} \label{eq:5}
u(0,t) &= \hat u^0(t) \qquad \text{at } x=0,
\end{align}
while the condition at $x=\ell$ becomes obsolete. 
For small $\eps>0$, the second boundary condition in \eqref{eq:2} therefore gives rise to a boundary layer at the outflow boundary $x=\ell$. 
In general, the solutions of \eqref{eq:1}--\eqref{eq:2} may also exhibit initial layers, whose presence can however be avoided by appropriate choice of initial values.

The asymptotic limit of convection-diffusion problems as $\eps \to 0$ has been studied intensively in the literature, both from an analytical and a numerical point of view; for details, one may refer e.g. to  \cite{Goering83, Linss, Miller12, RoosStynesTobiska, Shishkin}.
Problems with other types of boundary conditions have been considered, e.g., in \cite{ChaconRebollo10}.
For appropriate initial and boundary data, the solutions of \eqref{eq:1}--\eqref{eq:2} and \eqref{eq:4}--\eqref{eq:5} can be shown to satisfy the asymptotic estimate
\begin{align} \label{eq:est}
\|u_\eps(\cdot,t) - u(\cdot,t)\|_{L^2(0,\ell)} \le C \sqrt{\eps},
\end{align}
with a constant $C$ independent of $\eps$ and $t$. By considering the corresponding stationary problem, the rate $\sqrt{\eps}$ can also be seen to be optimal.

In this paper, we consider convection--diffusion problems on a one-dimensional network of pipes. In that case, equations \eqref{eq:1} and \eqref{eq:4} are assumed to hold for every single pipe while the boundary conditions \eqref{eq:2} and \eqref{eq:5} have to be augmented by appropriate coupling conditions at pipe junctions. These have to be chosen in order to guarantee conservation of mass across network junctions as well as dissipation of a mathematical energy, which is utilized to ensure the well-posedness of the problems.
We refer to \cite{Garavello, Lagnese, Mehmeti, Mugnolo} for background material on the analysis of partial-differential equations on networks.

The main result of our paper will be to prove that an estimate analogous to \eqref{eq:est} also holds for singularly perturbed convection--diffusion problems on networks. 
One of the main difficulties in the asymptotic analysis here is that the number and type of coupling conditions changes in the singular limit $\eps \to 0$. This gives rise to additional internal layers at pipe junctions that need to be handled appropriately.  
Since the nodal values $\hat u_\eps$, $\hat u$ in equations \eqref{eq:2} and \eqref{eq:5} are part of the solution and not prescribed a-priori, like the boundary values on a single pipe, a somewhat delicate choice of boundary layer functions at network junctions is required.

\bigskip 

The remainder of the manuscript is organized as follows:
In Section~\ref{sec:prelim}, we introduce our basic notation and then study the convection-diffusion and the transport problem on networks. The choice of suitable coupling conditions ensures conservation of mass at network junctions and dissipation of a mathematical energy, which in turn allows us to establish well-posedness of the problems by semigroup theory. 
In Section~\ref{sec:asymptotic}, we state and prove our main result, namely a quantitative estimate similar to \eqref{eq:est} for the convergence of solutions to the convection-diffusion problem with vanishing diffusion $\eps \to 0$ towards that of the corresponding transport problem.

\section{Notation and preliminaries} \label{sec:prelim}

After introducing our basic notation, we formally state the convection-diffusion and the limiting transport problem and study their well-posedness.

\subsection{Basic notation}

Following \cite{EggerKugler18}, the network is represented by a finite, directed, and connected graph with vertices $\V=\{v_1,\dots,v_n\}$ and edges $\E=\{e_1,\dots,e_m\}\subset \V\times\V$. For every edge $e=(v_i,v_j)$ we define two numbers
\begin{align}\label{eq:nev}
n^{e}(v_i)=-1\quad\text{and}\quad n^{e}(v_j)=1
\end{align} 
to indicate the start and end point of the edge, 
and we set $n^{e}(v)=0$ if $v\in\V\backslash\{v_i,v_j\}$. 
For any $v\in\V$ we define the set of incident edges {$\E(v):=\{e\in\E:n^{e}(v)\neq 0\}$}, and distinguish between inner vertices {$\V_0:=\{v\in\V:\vert\E(v)\vert\geq 2\}$} and boundary vertices $\V_{\partial}:=\V\backslash\V_0$; see Fig.~\ref{fig:network} for an illustration. 
\begin{figure}[ht]
\centering
\begin{tikzpicture}[scale=3]
\node (A) at (0,0.5) [circle,draw,thick] {$v_1$};
\node (B) at (0,-0.5) [circle,draw,thick] {$v_2$};
\node (C) at (0.75,0) [circle,draw,thick] {$v_3$};
\node (D) at (1.75,0) [circle,draw,thick] {$v_4$};

\draw[->, very thick] (A) to node[above] {$e_1$} (C);
\draw[->, very thick] (B) to node[above] {$e_2$} (C);
\draw[->, very thick] (C) to node[above] {$e_3$} (D);
\end{tikzpicture}
\caption{A network with three edges $e_1=(v_1,v_2)$, $e_2=(v_2,v_3)$, and $e_3=(v_2,v_4)$, inner vertex $\V_0=\{v_3\}$, and boundary vertices $\V_\partial=\{v_1,v_2,v_4\}$.
The set $\mathcal{E}(v_3)=\{e_1,e_2,e_3\}$ denotes the edges adjacent to the junction $v_3$.
Let the arrows depict the flow direction. Then we split the set of boundary vertices by $\V_{\partial}^{in}=\{v_1,v_2\}$ and $\V_\partial^{out}=\{v_4\}$ into inflow and outflow vertices. 
In a similar manner, we can split the set $\E(v_3)$ by $\E^{in}(v_3)=\{e_1,e_2\}$ and $\E^{out}(v_3)=\{e_3\}$ into edges that go into or out of the vertex $v_3$.}
\label{fig:network}
\end{figure}
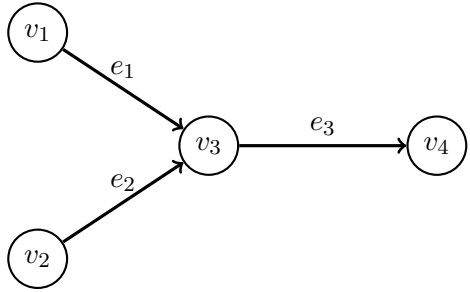

Every edge $e\in\E$ has a positive length $\ell^{e}$, 
and we identify $e$ with the interval $(0,\ell^{e})$. 
The Lebesgue measure on $(0,\ell^e)$ then induces a metric on $e$, and we denote by $L^2(e)=L^2(0,\ell^{e})$ the space of square integrable functions on the edge $e$. 
We further use
\begin{align*}
L^2(\E)=L^2(e_1)\times\dots\times L^2(e_m)=\{u: u^{e}\in L^2(e)\ \text{for all}\ e\in\E\}
\end{align*}
to denote the space of square integrable functions on the network. 
Here and below, $u^{e}=u\vert_e$ is the restriction of a function $u$ defined on the whole network to a single edge $e$. 
The natural norm and scalar product of the space $L^2(\E)$ are given by
\begin{align*}
  \| u \|_{L^2(\E)}^2=\sum\nolimits_{e\in\E}\Vert u^{e}\Vert_{L^2(e)}^2
\qquad\text{and}\qquad
 (u,w)_{L^2(\E)}=\sum\nolimits_{e\in\E}(u^{e},w^{e})_{L^2(e)}.
\end{align*}
%
We will further make use of the broken Sobolev spaces 
\begin{align*}
H_{pw}^s(\E)=\{u\in L^2(\E): u^{e}\in H^s(e) \text{ for all } e\in\E\},
\end{align*}
which are again equipped with the canonical norms and scalar products, defined by
\begin{align*}
  \| u\|_{H^s_{pw}(\E)}^2=\sum\nolimits_{e\in\E}\Vert u^{e}\Vert_{H^s(e)}^2
\qquad\text{and}\qquad
  (u,w)_{H^s_{pw}(\E)}=\sum\nolimits_{e\in\E}(u^{e},w^{e})_{H^s(e)}.
\end{align*}
Note that for $s>1/2$, the functions $u\in H^s_{pw}(\E)$ are continuous along edges $e\in\E$, while they may be discontinuous across junctions $v\in\V_0$. 
The subspace of functions that are also continuous across junctions is denoted by $H^1(\E)$. Elements of $H^1(\E)$ have a unique value $u(v)$ for every vertex $v \in \V$,
and we write $\ell_2(\V)$ for the set of  possible vertex values.

\subsection{Convection-diffusion problem}

We now formally introduce the convection-diffusion problem on networks to be studied, as well as our basic assumptions on the model parameters. 
A similar problem has been considered in \cite{Oppenheimer00}; also see \cite{Mugnolo} for further examples.

The transport of the substance along every edge $e \in \E$ shall be described by 
\begin{align} \label{eq:cd1}
a^e\dt u_\eps^e(x,t)+b^e\dx u_\eps^e(x,t)-\eps^e\dxx u_\eps^e(x,t)&=0, \qquad x\in e, \ e \in \E,
\end{align}
and we assume the concentration $u$ to be continuous across vertices, i.e.,
\begin{align} \label{eq:cd2}
u_\eps^e(v,t)&=\hat u_\eps^v(t), \qquad v\in\V,\ e\in\E(v), 
\end{align}
for some auxiliary functions $\hat u_\eps^v(t)$, $v \in \V$ to be determined by the following additional coupling conditions:
At pipe junctions $v \in \V_0$, we require that 
\begin{align}\label{eq:cd3}
\sum\nolimits_{e\in\E(v)}\big(b^eu_\eps^e(v,t)-\eps^e\dx u_\eps^e(v,t)\big)n^e(v)&=0, \qquad \ v\in\V_0,
\end{align}
and at boundary vertices $v \in \V_\partial$, we explicitly prescribe the concentration by 
\begin{align}\label{eq:cd4}
\hat u_\eps^v(t) &= g^v(t),\qquad v\in\V_\partial. 
\end{align}
The above equations are considered for $t>0$ and complemented by initial conditions 
\begin{align} \label{eq:cd5}
u_\eps^e(x,0)&=u_0^e(x), \qquad x\in e, \ e \in \E.
\end{align}
For the analysis of the convection--diffusion problem \eqref{eq:cd1}--\eqref{eq:cd5} which is developed in the rest of the paper, 
we make the following assumptions on the model parameters. 
\begin{assumption}\label{ass:1}
On every edge $e \in \E$, the functions $a$, $\eps$, and $b$ are constant and uniformly positive, and at pipe junctions $v \in \V_0$, 
the flow rate satisfies the conservation condition
\begin{align} \label{eq:b} 
\sum\nolimits_{e\in\E(v)}b^e n^e(v)&=0,  \qquad v\in\V_0. 
\end{align}
We further assume that the diffusion coefficient is bounded by $0 < \eps \le 1$.
\end{assumption}
\begin{remark}
The assumption that $a$ and $\eps$ are piecewise constant could be relaxed with minor changes in the arguments.
Since the flow direction changes when changing the orientation of the edge $e$, the sign of $b$ can always be adopted as desired by appropriate orientation of the edges. The basic assumption on $b$ therefore is, that is does no vanish. Otherwise, the transport problem \eqref{eq:4} degenerates to an ordinary differential equation. 
\end{remark}
The following theorem establishes well-posedness of the problem under consideration. 
\begin{theorem}\label{thm:cd}
Let Assumption~\ref{ass:1} hold and $T>0$. 
Then for any $u_0\in H^1(\E) \cap H_{pw}^2(\E)$ and $g\in C^2([0,T];\ell_2(\V_\partial))$ satisfying 
\eqref{eq:cd2}--\eqref{eq:cd4} for some $\hat u_0 \in \ell_2(\V)$,  
the system \eqref{eq:cd1}--\eqref{eq:cd5} 
has a unique classical solution 
\begin{align*}
u_\eps\in C^1([0,T];L^2(\E))\cap C^0([0,T];H^1(\E) \cap H^2_{pw}(\E))
\end{align*} 
with $\hat u_\eps^v(t) = u_\eps(v,t)$ defined by \eqref{eq:cd2}.  
Moreover, any solution of \eqref{eq:cd1}--\eqref{eq:cd5} satisfies 
\begin{align*}
\ddt \int_\E a u_\eps\, dx 
&= \sum_{v\in\V_\partial} \big(-b^e g^v+\eps^e \dx u_\eps(v)\big)n^e(v),
\end{align*}
i.e., mass is conserved up to flow over the boundary, as well as the  energy identity
\begin{align*}
\tfrac{1}{2} \ddt \| a^{1/2} u_\eps \|^2_{L^2(\E)}
&=-\| \eps^{1/2} \dx u_\eps \|^2_{L^2(\E)} 
  + \sum_{v\in\V_\partial}  \big(- \tfrac{1}{2} b^e g^v+\eps^e \dx u_\eps(v)\big)g^v\, n^e(v).
\end{align*}
\end{theorem}
\begin{proof}
For later reference we sketch the main arguments, which allow to apply the Lumer-Phillips theorem of semigroup theory; related results can also be found in \cite{Dorn08,KramarSikolya05,Mugnolo}. 

\emph{Step~1.}
Let $w(t) \in H^1(\E) \cap H^2_{pw}(\E)$ be the unique function that is affine linear on every edge and satisfies $w(v,t)=g^v(t)$ for all $v \in \V_\partial$ as well as $w(v,t)=0$ for $v \in \V_0$. 
Then any solution of the problem can be split into $u_\eps=w-z$ with $z(v,t)=0$ for all $v \in \V_\partial$, $t>0$, and using the linearity of the problem, one can see that $z$ satisfies 
\begin{align}
  \dt  z^{e}+b^{e}\dx z^{e}-\eps^e\dxx z^e=f^e,\qquad x\in e, \ e \in \E, \ t>0,\label{eq:cdh1}
\end{align}
with right hand side $f^e = \dt  w^{e}+b^{e}\dx w^{e}$, as well as the coupling conditions
\begin{align}
  z^{e}(v,t)&=\hat z^v(t),\qquad\ \ v\in\V,\ e\in\E(v),\ t>0. \label{eq:cdh2}
\end{align}
The auxiliary functions $\hat z^v(t) = \hat u^v(t)$ are here defined by the conservation condition
\begin{align} \label{eq:cdh3}
\sum\nolimits_{e\in\E(v)}\big(b^e z^e(v,t)-\eps^e\dx z^e(v,t)\big)n^e(v)&= 0,\qquad v\in\V_0,
\end{align}
at pipe junctions $v \in \V_0$, and by homogeneous boundary conditions 
\begin{align}\label{eq:cdh4}
\hat z^v(t) &= 0,\qquad v\in\V_\partial,
\end{align}
for the remaining vertices $v \in \V_\partial$. 
In addition, there holds
\begin{align} \label{eq:cdh5}
z^e(x,0)&=z_0^e(x), \qquad x\in e, \ e \in \E,
\end{align}
with $z_0^e(x) = w^e(x,0) - u_0^e(x)$. Let us note that by construction and the regularity assumption on $u_0$, we have $z_0 \in H^1(\E) \cap H^2_{pw}(\E)$ and $z_0(v)=0$ for all $v \in \V_\partial$. 

\emph{Step~2.}
Now set $\X=L^2(\E)$ with norm and scalar product defined by
\begin{align*}
  \| u\|_\X:=\| a^{1/2} u\|_{L^2(\E)} 
\qquad\text{and}\qquad 
(u,w)_\X:=(au,w)_{L^2(\E)}. 
\end{align*}
We further introduce the dense subspace 
\begin{align*}
 \D(\A_\eps):=\{z\in H^2_{pw}(\E): z \text{ satisfies  \eqref{eq:cdh2}--\eqref{eq:cdh4}} \text{ with some } \hat z \in \ell_2(\V)\},
\end{align*}
on which we formally define the linear operator
\begin{align}
  \A_\eps:\D(\A_\eps) \subset \X \rightarrow \X, 
\qquad \A_\eps z|_{e}&:=-\frac{1}{a^e}\big(b^e\dx z^e -\eps^e\dxx z^e\big). \label{eq:A}
\end{align}
%
Problem \eqref{eq:cdh1}--\eqref{eq:cdh5} can then be written as
an abstract evolution problem in $\X$, namely 
\begin{align}
  \dt z(t) &= \A_\eps z(t)+f(t),\qquad t>0,\label{eq:cdh6}\\
  z(0)&=z_{0}.\label{eq:cdh7}
\end{align}
By construction of $f$ and $z_0$ and the assumptions on the data, one can immediately see that $f \in C^1([0,T];\X)$ and $z_0 \in \D(\A_\eps)$. 
Moreover, the operator $\A_\eps$ satisfies 
\begin{align*} 
&(\A_\eps z, z)_\X 
= (-b \dx z + \eps \dxx z, z)_{L^2(\E)} 
= \sum\nolimits_{e\in\E} (-b^e \dx z^e + \eps^e \dxx z^e, z^e)_{L^2(e)} \\
&= \sum\nolimits_{e\in\E} (b^e z^e - \eps^e \dx z^e, \dx z^e)_{L^2(e)} + \sum\nolimits_{v \in \V} \sum\nolimits_{e\in\E(v)} \big(- b^e z^e(v) + \eps^e \dx z^e(v) \big) z^e(v) n^e(v). 
\end{align*}
The first term in the last line can be estimated by
\begin{align*}
(i) 
&= \sum\nolimits_{e\in\E} (b^e z^e - \eps^e \dx z^e, \dx z^e)_{L^2(e)} \\
&= \sum\nolimits_{v \in \V} \sum\nolimits_{e\in\E(v)} \tfrac{1}{2} b^e |z^e(v)|^2 n^e(v) - \sum\nolimits_{e \in \E}\eps^e \|\dx z^e\|^2_{L^2(e)} 
=(iii)+(iv).
\end{align*}
By rearranging the order of summation and use of the coupling and boundary conditions specified in \eqref{eq:cdh2}--\eqref{eq:cdh4} as well as the conservation condition \eqref{eq:b} for the flow rates, one can see that 
$(iii) =  \frac{1}{2} \sum\nolimits_{v \in \V} |\hat z^v|^2 \sum\nolimits_{e\in\E(v)} b^e n^e(v) = 0$, and hence 
\begin{align*}
(i) = (iv) = - \|\eps^{1/2} \dx z\|^2_{L^2(\E)}.
\end{align*}
The second term in the above expression for $(\A_\eps z, z)_\X$ can be further evaluated by
\begin{align*}
(ii) 
&= \sum\nolimits_{v \in \V} \sum\nolimits_{e\in\E(v)} \big(- b^e z^e(v) + \eps^e \dx z^e(v) \big) z^e(v) n^e(v) \\
&=\sum\nolimits_{v \in \V} \hat z^v \sum\nolimits_{e \in \E(v)}  \big(- b^e z^e(v) + \eps^e \dx z^e(v) \big) n^e(v) = 0,  
\end{align*}
where we again used the coupling and boundary conditions \eqref{eq:cdh2}--\eqref{eq:cdh4} appearing in the definition of the space $\D(\A_\eps)$. In summary, we thus have shown that
\begin{align} \label{eq:Aeps}
(\A_\eps z,z)_\X \le -  \|\eps^{1/2} \dx z\|^2_{L^2(\E)} \le 0 \qquad \text{for all}\ z \in \D(\A_\eps),
\end{align}
from which we deduce that $\A_\eps : \D(\A_\eps) \subset \X \to \X$ is dissipative. 
By the Lumer-Phillips theorem, the operator $\A_\eps$ thus is the generator of a strongly-continuous semigroup which implies the existence of a unique classical solution $z \in C^1([0,T];\X) \cap C([0,T];\D(\A_\eps))$ for \eqref{eq:cdh6}--\eqref{eq:cdh7}; see e.g. \cite{EngelNagel,Pazy} for details. 

\emph{Step~3.}
By combination with the regularity estimate for $w$ constructed in Step~1, one can see that $u=w-z$ is a solution to \eqref{eq:cd1}--\eqref{eq:cd5} with the required regularity.
Uniqueness follows by observing that the difference $z = u_1-u_2$ of any two solutions of \eqref{eq:cd1}--\eqref{eq:cd5} would solve \eqref{eq:cdh6}--\eqref{eq:cdh7} with $f=0$ and $z_0=0$, which implies $u_1-u_2 \equiv 0$.

\emph{Step~4.} 
Mass conservation follows by integrating \eqref{eq:cd1} over all pipes, summing up, 
and using the coupling and boundary conditions \eqref{eq:cd2}--\eqref{eq:cd4} as well as the balance condition \eqref{eq:b} for the flow rates. 
To show the energy-identity, we first multiply \eqref{eq:cd1} with $u^e$, integrate over the edges $e$, sum up the results, and use the coupling and boundary conditions \eqref{eq:cd2}--\eqref{eq:cd4}. Similar arguments were used to establish dissipativity of the operator $\A_\eps$ above.
\end{proof}

\begin{remark} \label{rem:cd}
The energy identity of Theorem~\ref{thm:cd} yields uniform bounds
\begin{align*}
\tfrac{1}{2} \|a^{1/2}u_\eps\|_{L^\infty(0,T;L^2(\E))} + \|\eps^{1/2} \dx u_\eps\|_{L^2(0,T;L^2(\E))} \le C(u_0,g),
\end{align*}
which allow to deduce existence and uniqueness of solutions also for less regular boundary and initial data. 
Similar results could be established alternatively also by Galerkin approximation; see \cite[Ch.~7]{Evans} or \cite[Ch.~XVIII]{DL5}.
\end{remark}

\subsection{Limiting transport problem}

We now turn to the vanishing diffusion limit $\eps \to 0$.
Since we assumed $b^e>0$ on every edge $e=(v_1,v_2)$, 
it is natural to call $v_1$ the inflow and $v_2$ the outflow vertex of the edge. 
For any $v \in \V$, we denote by $\E^{in}(v)=\{e \in \E: e=(\cdot,v)\}$ and $\E^{out}(v)=\{e \in \E: e=(v,\cdot)\}$ the edges which carry flow into or out of the vertex $v$, and we further split the boundary vertices into the sets $\V_\partial^{in} = \{v \in \V_\partial: |\E^{out}(v)|=1\}$ and $\V_\partial^{out}=\{v \in \V_\partial: |\E^{in}(v)|=1\}$; see Fig.~\ref{fig:network} for an illustration.
 We then consider the following problem; see \cite{DornKramarNagelRadl10,EggerPhilippi20} for related results. 
On every edge $e \in \E$, the transport is described by 
\begin{align}
a^e\dt u^e(x,t)+b^e\dx u^e(x,t)&=0, \qquad x \in e, \ e \in \E. \label{eq:t1}
\end{align}
In contrast to the convection-diffusion problem, we now only need one boundary condition at the inflow boundary of each edge, and accordingly we set
\begin{align}\label{eq:t2}
u^e(v,t)&=\hat u^v(t), \qquad v\in\V 
,\ e\in\E^{\text{out}}(v),
\end{align}
with auxiliary values $\hat u^v$ determined by the conservation condition
\begin{align}\label{eq:t3}
\sum\nolimits_{e\in\E^{in}(v)} b^e u^e(v,t) n^e(v) + \sum\nolimits_{e\in\E^{out}(v)} b^e \hat u^v(t) n^e(v) &= 0, \qquad v \in \V_0,
\end{align}
at inner vertices.
Note that the vertices in $\V_0$ have at least one inflow and one outflow edge. 
On the inflow boundary vertices, which only have one outflow edge, we set 
\begin{align}\label{eq:t4}
  \hat u^v(t)&=g^v(t),\qquad v\in\V_\partial^{in}.
\end{align}
The above equations are assumed to hold for $t>0$ and complemented by initial conditions
\begin{align} \label{eq:t5}
u^e(x,0)&=u_0^e(x), \qquad x \in e, \ e \in \E.
\end{align}
%
From equation~\eqref{eq:t3} and the conservation condition~\eqref{eq:b} for the flow rate $b$, one can deduce that the nodal values $\hat u^v$ at inner vertices $v \in \V_0$ are convex combinations of the concentrations $u^e(v)$, $e \in \E^{in}(v)$ entering the junction $v$. 
These \emph{mixtures} serve as inflow values for the pipes $e \in \E^{out}(v)$ with flow leaving the corresponding vertex.
%

\begin{remark}
For the asymptotic analysis given in Section~\ref{sec:asymptotic}, it will be convenient to additionally define values $\hat u^v$ for the outflow vertices by
\begin{align}\label{eq:t6}
  \hat u^v(t)&=g^v(t),\qquad v\in\V_\partial^{out},
\end{align}
where $g^v$, $v \in \V_\partial$ are the same boundary data 
as for the convection-diffusion problem.
Note that the values $\hat u^v$, $v \in \V_\partial^{out}$ do not appear in the other equations, and therefore are not required for the analysis of the transport problem presented in the sequel.
\end{remark}

With similar arguments as in the analysis of the convection--diffusion problem \eqref{eq:cd1}--\eqref{eq:cd5},  
we can also obtain a well-posedness result for the transport problem \eqref{eq:t1}--\eqref{eq:t5}.
\begin{theorem}\label{thm:t}
Let Assumption~\ref{ass:1} hold and $T>0$ be given. 
Then for any $u_0\in H_{pw}^1(\E)$ and $g\in C^2([0,T];\ell_2(\V_{\partial}^{in}))$, 
satisfying \eqref{eq:t2}--\eqref{eq:t4} at $t=0$ with some $\hat u_0\in\ell_2(\V\setminus\V_\partial^{out})$, 
the system \eqref{eq:t1}--\eqref{eq:t5} has a unique classical solution
\begin{align*}
u\in C^1([0,T];L^2(\E))\cap C^0([0,T];H^1_{pw}(\E))
\end{align*}
with $\hat u \in C^0([0,T];\ell_2(\V\setminus\V_\partial^{out}))$ defined by \eqref{eq:t2}. 
Moreover, the solution satisfies 
\begin{align*}
\ddt \int_\E a u\, dx 
&= \sum\nolimits_{v\in\V_\partial^{in}} b^e g^v - \sum\nolimits_{v\in\V_\partial^{out}} b^e u^e(v),
\end{align*}
i.e., mass is conserved up to flow over the boundary,
as well as the energy identity 
\begin{align*}
\ddt \| a^{1/2} u \|^2_{L^2(\E)}
&= 
 \sum\nolimits_{v\in\V_\partial^{\text{in}}} b^e |g^v|^2 
 -\sum\nolimits_{v\in\V_\partial^{\text{out}}} b^e |u^e(v)|^2
\\&\qquad\qquad\quad 
-\sum\nolimits_{v\in\V_0}\sum\nolimits_{e\in\E^{\text{in}}(v)} b^{e} \, |u^{e}(v)-\hat u^{v}|^2. 
\end{align*}
\end{theorem}
\begin{proof}
One can proceed with similar arguments as in the proof of Theorem \ref{thm:cd}, and we therefore only sketch the basic steps and the main differences. 

\emph{Step~1.}
The solution can again be split into two parts $u=w-z$ where $w(t)$, $t>0$ is a prescribed piecewise linear function in space that satisfies the inflow boundary conditions as well as $w^e(v)=0$ for all $v\in\V\setminus\V_\partial^{in},\ e\in\E(v)$, and the function $z$ satisfies the equations with inhomogeneous right hand side and zero inflow boundary conditions. 

\emph{Step~2.} 
We set $\X=L^2(\E)$ as before and define the dense subspace  
\begin{align*}
\D(\A)=\{z \in H^1_{pw}(\E) : z &\text{ satisfies \eqref{eq:t2}--\eqref{eq:t3} for some } \hat z \in \ell_2(\V \setminus \V_\partial^{out})\\
&\text{ with } \hat z^v=0 \text{ for } v\in\V_\partial^{in} \},
\end{align*}
on which we formally define the linear operator 
\begin{align*}
\A : \D(\A) \subset \X \to \X, 
\qquad \A z|_e = - \frac{1}{a^e} b^e \dx z^e.
\end{align*}
The transport problem \eqref{eq:t1}--\eqref{eq:t5} can then be written as an abstract evolution problem 
\begin{align}
\dt z(t) &= \A z(t) + f(t), \qquad t>0, \\ 
    z(0) &= z_0,
\end{align}
with $f(t) = a \dt w(t) + b \dx w(t)$ and $z_0=w(0)-u_0$ given. 
Due to the choice of $w$ and the assumptions on the problem data, one can guarantee that $f \in C^1([0,T];\X)$ and $z_0 \in \D(\A)$. 
Using similar arguments as in the proof of Theorem~\ref{thm:cd}, one can further show that 
\begin{align*}
(\A z, z)_\X
&= -\sum\nolimits_{e \in \E}(b^e \dx z^e, z^e)_{L^2(e)}\\
&= -\tfrac{1}{2} \sum\nolimits_{e \in \E}  b^e (|z^e(v^e_{o})|^2 - |z^e(v_i^e)|^2),   
\end{align*}
where $v_i^e$, $v_o^e$ denote the inflow and outflow vertex of the edge $e=(v_i^e,v_o^e)$. By exchanging the order of summation and using the coupling conditions \eqref{eq:t2}, we then get
\begin{align*}
(\A z, z)_\X
&= \tfrac{1}{2} \sum\nolimits_{v \in \V} \left( \sum\nolimits_{e \in \E^{out}(v)} b^e |\hat z^v|^2 -  \sum\nolimits_{e \in \E^{in}(v)} b^e |z^e(v)|^2 \right).
\end{align*}
Using the fact that $\hat z^v$ for $v \in \V_0$ is a convex combination of the values $z^e(v)$, $e \in \E^{in}(v)$, we can estimate the first term in this identity by Jensen's inequality, which yields
\begin{align*}
\sum\nolimits_{e \in \E^{out}(v)} b^e |\hat z^v|^2 
\le \sum\nolimits_{e \in \E^{in}(v)} b^e |z^e(v)|^2  
\end{align*}
for all $v \in \V_0$. As a consequence, we obtain the inequality
\begin{align*}
(\A z,z)_\X \le \tfrac{1}{2}\sum\nolimits_{v \in \V_\partial^{in}} b^e |z^e(v)|^2-\tfrac{1}{2}\sum\nolimits_{v \in \V_\partial^{out}} b^e |z^e(v)|^2 \le 0.
\end{align*}
In the last inequality we used that $z$ vanishes at the inflow vertices $v\in\V_\partial^{in}$. 
From this estimate, we can again deduce that $\A$ is dissipative, and semigroup theory guarantees the existence of a unique classical solution $z \in C^1([0,T];X) \cap C([0,T];\D(\A))$.

\emph{Steps~3 and 4.}
The existence of a unique solution $u=w-z$ for problem~\eqref{eq:t1}--\eqref{eq:t5} is now established with the same arguments as in the proof of Theorem~\ref{thm:cd}. 
Mass conservation and energy--dissipation again follow by appropriate testing; see \cite{EggerPhilippi20} for details.
\end{proof}

\subsection{Comparison of the coupling conditions} \label{sec:t}

Before we proceed, let us briefly comment on the coupling conditions. 
For the convection--diffusion problem with $\eps>0$, the number of coupling conditions at a junction $v \in \V_0$ is $|\E(v)|+1$, which suffices to guarantee continuity of the solution and conservation of mass at the junction. 
For the transport problem, on the other hand, the number of coupling conditions is $|\E^{out}(v)|+1$ which only suffices to guarantee conservation of mass at the junction and to prescribe the concentrations at the outflow edges.
The concentration $u^e(v)$, $e \in \E^{in}(v)$ on edges with flows into the junctions will however usually deviate from the mixing value $\hat u^v$. 
In this case, the mixing at pipe junctions generates dissipation, which amounts to the inequality resulting from the application of Jensen's inequality in Step~2 of the proof of the previous lemma.

\section{Asymptotic analysis}\label{sec:asymptotic}

We will now show that the solutions  of the convection--diffusion problem \eqref{eq:cd1}--\eqref{eq:cd5} 
converge to that of the transport problem \eqref{eq:t1}--\eqref{eq:t5} with rate $\mathcal{O}(\sqrt{\eps})$. 
We will closely follow the arguments of the proof for the corresponding result for a single edge, which can be found in \cite[p.\,159--166]{RoosStynesTobiska}; see \cite{Bobisud67} for the original reference.
Following \cite{RoosStynesTobiska}, we start with establishing some preliminary results that will be required for the proof. 

\subsection{Auxiliary results}

As a first step, we establish a weak maximum principle for solutions of convection-diffusion problems on networks. 
\begin{lemma} \label{lem:max}
Let $u\in C^1([0,T];L^2(\E))\cap C^0([0,T];H^1(\E) \cap H^2_{pw}(\E))$ satisfy 
\begin{alignat}{4}
a^e \dt u^e+b^e \dx u^e - \eps^e \dxx u^e &\geq 0, \qquad && e\in\E, \label{eq:max1}\\
\sum\nolimits_{e\in\E(v)} \eps^e \dx u^e(v) n^e(v) &=0,  \qquad && v\in\V_0, \label{eq:max2}\\
u(v) &\geq 0, \qquad  && v\in\V_\partial, \label{eq:max3}
\intertext{for all $0 < t < T$, as well as the initial conditions}
u^{e}(x,0) &\geq 0, \qquad && x\in e,\ e\in\mathcal{E}.\label{eq:max4}
\end{alignat}
Then the function $u$ is non-negative, i.e., $u \ge 0$ on $\E$ for all $t\in[0,T]$. 
\end{lemma}
\begin{proof}
We multiply the differential inequality \eqref{eq:max1} by the test function $w:=\min(0,u)\leq 0$, integrate over all edges $e\in\E$, and use integration-by-parts for the spatial derivative terms, similar as in the proof of Theorem~\ref{thm:cd} and \ref{thm:t}. This leads to 
\begin{align*}
0 &\ge (a\dt u,w)_{L^2(\E)}+(b\dx u,w)_{L^2(\E)}-(\eps\dxx u,w)_{L^2(\E)} \\
  &= (a \dt u, w)_{L^2(\E)} - (b u, \dx w)_{L^2(\E)} + (\eps \dx u, \dx w)_{L^2(\E)},
\end{align*}
where we used continuity of $u$ and $w$ across junctions, the conservation condition \eqref{eq:b} for the flow rates, as well as \eqref{eq:max2} and the fact that $u \ge 0$ on the boundary, and hence $w=0$ at vertices $v \in \V_\partial$.
Next, observe that $w(t) \equiv 0$, and thus also $\dx w(t) \equiv 0$, on the set where $u$ is non-negative, 
and $w \equiv u$ on the complement $\E_-(t)=\{x : u(x,t) < 0\}$. 
From this and the previous inequality, we immediately deduce that
\begin{align*}
0 &\ge (a \dt u, u)_{L^2(\E_-)} - (b u, \dx u)_{L^2(\E_-)} + (\eps \dx u, \dx u)_{L^2(\E_-)} \\
  &\ge (a \dt u,u)_{L^2(\E_-)},
\end{align*}
where we used that $b u \dx u = \frac{b}{2} \dx |u|^2$, the positivity of $b$, and the fact that possible coupling and boundary terms appearing when integrating this expression drop out due to continuity of $u$ across junctions; furthermore, we employed the flow conservation condition of $b$ and the fact that $u=0$ on the boundary of $\E_-$ due to its definition and \eqref{eq:max3}. 
By integration in time and using $u=0$ on the boundary of $\E_-(t)$ as well as $u(0) \ge 0$, we then obtain 
\begin{align*}
0 \ge \int_0^t (a \dt u(s), u(s))_{\E_-(s)} ds 
  \ge \frac{1}{2} \int_{\E_-(t)} a |u(t)|^2 dx.
\end{align*}
This further implies that $\E_-(t)=\emptyset$, and hence $u(t) \ge 0$ for all $0 \le t \le T$.  
\end{proof}

Using the weak maximum principle, we can show the following uniform bounds.
\begin{lemma}\label{lem:bound}
Let Assumption~\ref{ass:1} hold. Then the solution of problem \eqref{eq:cd1}--\eqref{eq:cd5} is uniformly bounded by $|u_\eps(x,t)| + |\dt u_\eps(x,t)| \le C_u$ for all $x \in \E$, $t \in [0,T]$ with $C_u$ independent of $\eps$.
\end{lemma}
\begin{proof}
The boundedness of $u_\eps$ follows from the maximum principle with the usual arguments; see e.g., \cite[Ch.~7]{Evans}. 
By linearity of the problem, one can further see that $z_\eps=\dt u_\eps$ again solves \eqref{eq:cd1}--\eqref{eq:cd5}, but with  
with boundary data $z(v)=\dt g^v$ on $\V_\partial$ and initial data $z_\eps^e(0)=\dt u_\eps^e(0) = -\frac{1}{a^e} \left( b^e \dx u_0^e - \eps^e \dxx u_0^e \right)$. 
The boundedness of $z_\eps=\dt u_\eps$ then follows from the assumptions on the problem data with the same reasoning as above.
\end{proof}

\begin{lemma}\label{lem:boundderiv}
Let Assumption~\ref{ass:1} hold and $u_\eps$ denote the solution of problem \eqref{eq:cd1}--\eqref{eq:cd5}. Then 
\begin{align*}
|\dx u_\eps^e(v_i^e,t)| \le K \quad \text{for all } t\in(0,T),
\end{align*}
for all edges $e=(v_i^e,v_o^e)$ with uniform constant $K$ independent of $e \in \E$ and $\eps$.
\end{lemma}
\begin{proof}
For every edge $e=(v_i^e,v_o^e) \simeq (0,\ell^e)$, we define 
$w^e(x,t):=  K x + \hat u_\eps^{v_i^e}(t)-u_\eps^e(x,t)$,
where $K$ is a positive constant to be chosen later. 
From Lemma~\ref{lem:bound} we know that $u_\eps$ and $\dt u_\eps$, and hence by \eqref{eq:cd2} also $\hat u_\eps^v$ and $\dt \hat u_\eps^v$, are bounded independently of $\eps$ by a uniform constant $C_u$. 
Then for any 
$K\geq \frac{a^e}{b^e} C_u$, 
we have 
\begin{align*}
a^e \dt w^e + b^e \dx w^e - \eps^e \dxx w^e 
  = a^e \dt \hat u_\eps^{v_i^e} + b^e K  \geq 0.
\end{align*}
If we further assume that $K \geq \max_{x\in \E} |\dx u_0(x)|$, then 
\begin{align*}
w^e(x,0) 
&= K x+\hat u_\eps^{v_i^e}(0)-u_\eps^e(x,0) 
 = K x+u_0^e(0)-u_0^e(x)\\
&= K x - \int_0^x \dx u_0^e(s) ds 
 \geq Kx-\max\nolimits_{x\in e} |\dx u_0(x)| x
 \geq 0.
\end{align*}
Using Lemma~\ref{lem:bound}, we may further assume that $K \geq 2 C_u / \min_{e \in \E} \ell^e$, and deduce that
\begin{align*}
w^e(v_i^e,t)=0 
\qquad \text{and} \qquad 
w^e(v_o^e,t)=K\cdot \ell^e+\hat u_\eps^{v_i^e}(t)-u_\eps^e(v_o^e,t)\geq 0,
\end{align*}
since $u_\eps$ is assumed to be continuous across network junctions. The weak maximum principle then yields $w^e\geq 0$ for all $t\in [0,T]$, and consequently
\begin{align*}
u_\eps^e(x,t)-\hat{u}_\eps^{v_i^e}(t) \leq K x.    
\end{align*}
This implies that $|\dx u_\eps^e(v_i^e,t)|\leq K$ for all $t\in(0,T)$,
and from the construction, one can further see that the constant $K$ can be chosen independent of $e \in \E$ and of $\eps$.
\end{proof}

\subsection{Asymptotic estimates}

With the auxiliary results derived in the previous section, we are now in the position to prove the main result of the manuscript.

\begin{theorem} \label{thm:est}
Let Assumption~\ref{ass:1} hold. Further let $u_\eps$ be the solution of problem \eqref{eq:cd1}--\eqref{eq:cd5} 
and $u$ be the solution of the corresponding limit problem \eqref{eq:t1}--\eqref{eq:t6}. 
Then 
\begin{align} \label{eq:rate}
\| u_\eps-u\|_{L^\infty(0,T;L^2(\E))} \leq C \sqrt{\eps},
\end{align}
with a constant $C$ that is independent of the diffusion parameter $0<\eps\le 1$. 
\end{theorem}
\begin{proof}
The proof follows the arguments given in \cite[p.\,159--166]{RoosStynesTobiska}.  Since we require particular boundary layer functions for junctions $v \in \V_0$, 
we present the result in detail.

\emph{Step~1.} 
For every $e\in \E$ with $e=(v_i^e,v_o^e) \simeq (0,\ell^e)$, we define a boundary layer function 
\begin{align}\label{eq:blf}
w_\eps^e(x,t) = \left( \hat u ^{v_o^e}(t)-u^e(v_o^e,t) \right) e^{-b^e (\ell^e-x)/\eps^e}.
\end{align} 
%
From this particular construction, we immediately obtain
\begin{align}\label{eq:blf-pde}
b^e \dx w_\eps^e - \eps^e \dxx w_\eps^e = 0, 
\end{align}
and $\| w_\eps\|_{L^\infty(0,T;L^2(\E))} \leq C\sqrt{\eps}$, 
where we used that $u$, and thus also $\hat u^v$, are uniformly bounded according to Theorem~\ref{thm:t}. 
Further estimates for $w_\eps$ and its spatial derivatives can be found in \cite{Dobrowolski97}. 
The error between $u_\eps$ and $u$ can then be split into 
\begin{align*}
\| u_\eps-u \|_{L^\infty(0,T;L^2(\E))} 
&\leq  \| u_\eps - u -w_\eps \|_{L^\infty(0,T;L^2(\E))} + 
 \| w_\eps \|_{L^\infty(0,T;L^2(\E))}\\
&\leq \| u_\eps - u - w_\eps \|_{L^\infty(0,T;L^2(\E))} + C \sqrt{\eps}.
\end{align*}

\emph{Step~2.} 
For ease of notation, we introduce $\eta_\eps:=u_\eps-u-w_\eps$ and investigate the values of $\eta_\eps$ at time $t=0$ and at the vertices $v \in \V$ of the network.
For $t=0$, we have
\begin{align} \label{eta:init}
\eta_\eps^e(x,0)
&= u_\eps^e(x,0) - u^e(x,0) - \big( \hat u^{v_o^e}(0) - u^e(v_o^e,0) \big) e^{-b^e (\ell^e-x)/\eps^e} \\
&= u_0^e(x) - u_0^e(x) - \big( u^e_0(v_o^e) - u_0^e (v_o^e) \big) e^{-b^e (\ell^e-x)/\eps^e} 
 = 0,\nonumber
\end{align}
where we used that $u_\eps$ and $u$ have the same initial value $u_0$ which is continuous across junctions $v\in\mathcal{V}_0$ and  $g^{v}(0)=u_0(v)$ for $v \in \V_\partial$ due to the compatibility conditions of initial and boundary values. 
For inflow boundary vertices $v \in \V_\partial^{in}$ and $e=(v,v_o^e)$, we obtain 
\begin{align}\label{eta:inflow}
\eta_\eps^e(v,t)
&=g^v(t) - g^{v}(t) - \big(\hat u^{v_o^e}(t) - u^e(v_o^e,t)\big) e^{-b^e \ell^e/\eps^e} \leq C' \eps,
\end{align}
where $C'$ is a constant independent of $\eps$.
For outflow boundary vertices $v\in\V_\partial^{out}$ and the corresponding edge $e=(v_i^e,v)$, we obtain 
\begin{align}\label{eta:outflow}
\eta_\eps^e(v,t)
= g^{v}(t) - u^e(v,t) - \big( g^{v}(t) - u^e(v,t) \big)
= 0.
\end{align}
At inner vertices $v\in\mathcal{V}_0$, on the other hand, there holds
\begin{alignat}{5}
\eta_\eps^e(v,t) &= \hat{u}_\eps^v(t)-\hat u^v(t), \quad && e=(v_i^e,v) \in \E^{in}(v), \label{eta:inner-in} \\
\eta_\eps^e(v,t) &= \hat u_\eps^v(t) - \hat u^v(t) - \big( \hat u^{v_o^e}(t) - u^e(v_o^e,t) \big) \, e^{-b^e \ell^e/\eps^e}, \quad &&  e=(v,v_o^e) \in \E^{out}(v).\label{eta:inner-out}
\end{alignat}

\emph{Step~3.}
Inserting $\eta_\eps$ into the convection diffusion equation \eqref{eq:cd1} and testing with $\eta_\eps$ yields
\begin{align*}
(a \dt \eta_\eps,\eta_\eps)_{L^2(\E)}
&=-(b \dx \eta_\eps,\eta_\eps)_{L^2(\E)}+ (\eps \dxx \eta_\eps,\eta_\eps)_{L^2(\E)}
+ (\eps\dxx u,\eta_\eps)_{L^2(\E)} \\&\qquad \qquad \qquad \qquad \, \;
- (a\dt w_\eps,\eta_\eps)_{L^2(\E)}
= (i) + (ii) + (iii) + (iv),
\end{align*}
where we used the identity \eqref{eq:blf-pde}. 
The individual terms are now estimated separately.

\emph{Step~3(i).}
The first term can be transformed into
\begin{align*}
(i) 
 = -\sum\nolimits_{v\in\V} \sum\nolimits_{e\in\E(v)} \tfrac{1}{2} b^e |\eta_\eps^e(v)|^2 n^e(v)
 = \sum\nolimits_{v \in \V} (*) .
\end{align*}
For internal vertices $v\in\V_0$, using \eqref{eta:inner-in}--\eqref{eta:inner-out} we obtain 
\begin{align*}
(*)
&= \sum\nolimits_{e\in\E^{out}(v)} \tfrac{1}{2} b^e \big(\hat u_\eps^v - \hat u ^v - \big( \hat u^{v_o^e} - u^e(v_o^e) \big) e^{-b^e \ell^e/\eps^e}\big)^2
\\
& \qquad \qquad 
-\sum\nolimits_{e\in\E^{in}(v)} \tfrac{1}{2} b^e (\hat u_\eps^v - \hat u ^v )^2\\
&= \sum\nolimits_{e\in\E^{out}(v)} \tfrac{1}{2} b^e \big(\hat u_\eps^v - \hat u^v\big)^2 - \sum\nolimits_{e\in\E^{in}(v)}\tfrac{1}{2} b^e \big(\hat u_\eps^v-\hat u^v\big)^2 \\
& \qquad \qquad -\sum\nolimits_{e\in\E^{out}(v)} b^e \big(\hat u_\eps^v - \hat u^v\big) \big(\hat u^{v_o^e} - u^e(v_o^e)\big) e^{-b^e \ell^e/\eps^e}\\
& \qquad \qquad +\sum\nolimits_{e\in\E^{out}(v)} \tfrac{1}{2} b^e \big(\hat u^{v_o^e} - u^e(v_o^e) \big)^2 e^{-2 b^e \ell^e/\eps^e}\\
&\leq C \sum\nolimits_{e\in\E^{out}(v)} b^e (e^{-b^e \ell^e/\eps^e} + e^{-2 b^e \ell^e/\eps^e}) 
\leq C'\eps.
\end{align*}
Here we additionally used the conservation property \eqref{eq:b} of the volume flow rates and the uniform boundedness of $u_\eps$ stated in Lemma~\ref{lem:bound}.
For inflow boundary vertices $v\in\V_\partial^{in}$, we know from \eqref{eta:inflow} that $\eta_\eps^e(v,t)\leq C\eps$ on $e=(v,v_o^e)$, and hence
$(*) 
\leq C b^e \eps$,
and for outflow boundary vertices $v\in\V_\partial^{out}$, we have 
$\eta_\eps^e(v,t)=0$ by \eqref{eta:outflow}, and thus $(*)=0$ there. 
In summary, we thus obtain 
$(i) \le C'' \eps$
with constant $C''$ independent of $\eps$. 

\emph{Step~3(ii).}
Using integration-by-parts, we can transform the second term into
\begin{align*}
(ii) 
&= -\sum\nolimits_{e \in \E} (\eps^e \dx \eta_\eps^e,\dx \eta_\eps^e)_{L^2(e)} + \sum\nolimits_{v \in \V} \sum\nolimits_{e \in \E(v)} \eps^e \dx \eta_\eps^e(v) \eta_\eps^e(v) n^e(v) \\ 
&\leq \sum\nolimits_{v \in \V} \sum\nolimits_{e \in \E(v)} \eps^e \dx \eta_\eps^e(v) \eta_\eps^e(v) n^e(v)  
 = \sum\nolimits_{v \in \V} (**).
\end{align*}
At inner vertices $v\in\V_0$, we again use \eqref{eta:inner-in}--\eqref{eta:inner-out} to obtain
\begin{align*}
(**) 
=&\sum\nolimits_{e\in\E(v)} \eps^e \dx u_\eps^e(v) (\hat u_\eps^v - \hat u^v) n^e(v) 
- \sum\nolimits_{e\in\E^{out}(v)} \eps^e \dx \eta_\eps^e(v) w_\eps^e(v) \\
&-\sum\nolimits_{e\in\E(v)} \eps^e (\dx u^e(v) + \dx w_\eps^e(v)) (\hat u_\eps^v - \hat u^v) n^e(v)
 = (a) + (b) + (c). 
\end{align*}
Due to \eqref{eq:cd3}, the term (a) vanishes. 
Inserting the definition of $\eta_\eps$, we further obtain
\begin{align*}
(b)	
&=-\sum\nolimits_{e\in\E^{out}(v)} \eps^e \big(\dx u_\eps^e(v) - \dx u^e(v) - \dx w_\eps^e(v)\big) w_\eps^e(v).
\end{align*}
From Lemma~\ref{lem:boundderiv}, we know that $\dx u_\eps^e(v)$ is bounded uniformly for all $e\in\mathcal{E}^{out}(v)$, and the derivative $\dx u^e$ is also bounded independently of $\eps$.  
Furthermore, the spatial derivative $\dx w_\eps^e(v)$ can be bounded by $(C/\eps^e) e^{-b^e \ell^e/\eps^e}$ for all $e\in\E^{out}(v)$; see  \eqref{eq:w-deriv-out}. 
From these bounds we conclude that $(b) \le C (\eps^e+1) e^{-b^e \ell^e/\eps^e} \leq C' \eps$ with constant $C'$ independent of $\eps$.
To estimate the term (c), we observe that 
$\dx u$ and $\hat u_\eps$ are bounded independently of $\eps$; see Theorem~\ref{thm:t} and Lemma~\ref{lem:bound}. 
Consequently,
\begin{align*}
(c_1) = -\sum\nolimits_{e\in\E(v)} \eps^e \dx u^e(v) ( \hat u_\eps^v - \hat u^v) n^e(v) 
\leq C \eps.
\end{align*}
For the spatial derivative $\dx w_\eps^e(v)$, we further obtain
\begin{alignat}{5}
\dx w_\eps^e(v)&=\frac{b^e}{\eps^e}(\hat u^v - u^e(v)), \qquad && e \in \E^{in}(v), \label{eq:w-deriv-in} \\
\dx w_\eps^e(v) &=\frac{b^e}{\eps^e}(\hat{u}^{v_o^e}-u^e(v_o^e)) e^{-b^e \ell^e/\eps^e}, \qquad && e \in \E^{out}(v), \label{eq:w-deriv-out}
\end{alignat}
which allows us to rewrite
\begin{align*}
(c_2) 
&= 
-\sum\nolimits_{e\in\E(v)} \eps^e \dx w_\eps^e(v) (\hat u_\eps^v - \hat u^v) n^e(v) \\
&=-\sum\nolimits_{e\in\E^{in}(v)} b^e (\hat u ^v - u^e(v)) (\hat u_\eps^v - \hat u^v) \\
&\qquad \qquad 
+\sum\nolimits_{e\in\E^{out}(v)} b^e (\hat u^{v_o^e}-u^e(v_o^e)) e^{-b^e \ell^e/\eps^e} (\hat u_\eps^v-\hat u ^v).
\end{align*}
Now, the first term on the right hand side vanishes due to the coupling conditions \eqref{eq:t2}--\eqref{eq:t3} and the conservation condition \eqref{eq:b} for the flow rates.
The uniform bounds for $u$, $\hat u^v$ and $u_\eps$, $\hat u_\eps^v$ then allow to bound $(c_2) \le C' \eps$, and hence $(c) \le C'' \eps$ with $C''$ independent of $\eps$.
By combination of the estimates for (a), (b), and (c), we obtain $\sum\nolimits_{v\in\V_0}(**) \le C \eps$. 
For the remaining boundary vertices $v\in\V_\partial$, we use \eqref{eta:inflow}--\eqref{eta:outflow} to see that
\begin{align*}
\sum\nolimits_{v\in\V_\partial}(**)=\sum\nolimits_{v\in\V_\partial^{in}} \eps^e\dx \eta_\eps^e(v)\eta^e(v) n^e(v) \leq C'\eps,    
\end{align*}
since $\dx \eta_\eps^e(v)$, $v\in\V_\partial^{in}$ is bounded independently of $\eps$ by Theorem \ref{thm:t}, Lemma \ref{lem:boundderiv}, and \eqref{eq:w-deriv-out}.
In summary, we thus obtain $(ii) \le C\eps$ with a constant $C$ independent of $\eps$.

\emph{Step~3(iii).}
Integration-by-parts and Young's inequality yield
\begin{align*}
(iii) 
&= -(\eps \dx u,\dx\eta_\eps)_{L^2(\E)} + \sum\nolimits_{v \in \V} \sum\nolimits_{e \in \E(v)} \eps^e \dx u^e(v)  \eta_\eps^e(v) n^e(v) \\
&\leq \tfrac{1}{2} \|\eps^{1/2} \dx u\|_{L^2(\E)}^2 +\tfrac{1}{2} \| \eps^{1/2} \dx \eta_\eps \|_{L^2(\E)}^2
+ \sum\nolimits_{v\in\V} \sum\nolimits_{e\in\E(v)} \eps^e \dx u^e(v) \eta_\eps^e(v) n^e(v).
\end{align*}
The first term is bounded by $C \eps$, the second term can be absorbed into (ii), and the boundary terms can be estimated by $C \eps$, since $\dx u$ and $\eta_\eps$ are uniformly bounded; see Theorem~\ref{thm:t} and \eqref{eta:inflow}--\eqref{eta:inner-out}. In summary, we thus obtain 
$(iii) \le C\eps$.

\emph{Step~3(iv).}
Using Young's inequality, we have
\begin{align*}
-(a\dt w_\eps,\eta_\eps)_{L^2(\E)} 
\leq \tfrac{1}{2} \|a^{1/2} \dt w_\eps\|_{L^2(\E)}^2 + \tfrac{1}{2} \|a^{1/2} \eta_\eps\|_{L^2(\E)}^2.
\end{align*}
By the uniform bounds for $\dt u$ and $\dt u_\eps$, 
we can estimate the first term by
\begin{align*}
\|a^{1/2} \dt w_\eps \|_{L^2(e)}^2 
&=\int_0^{\ell^e} a^e \big(\dt \hat u^{v_o^e}(t)-\dt u^e(v_o^e,t) \big)^2
e^{-2 b^e(\ell^e-x)/\eps^e}\, dx
\leq C'\eps,
\end{align*}
and since the graph is finite, this estimate translates to the whole 
network.

\emph{Step~4.}
By combination of the estimates for the terms $(i)$--$(iv)$, we finally obtain 
\begin{align*}
\tfrac{1}{2} \tfrac{d}{dt} \|a^{1/2} \eta_\eps\|_{L^2(\E)}^2 
&=(a \dt \eta_\eps,\eta_\eps)_{L^2(\E)}
 \leq C\eps + \tfrac{1}{2} \|a^{1/2} \eta_\eps\|_{L^2(\E)}^2.
\end{align*}
An application of Gronwall's lemma then immediately yields 
\begin{align*}
\|\eta_\eps(t)\|_{L^2(\E)} \leq 2 a_{\min}^{-1} C  e^{t} \eps \le C' \eps,
\end{align*}
with $a_{\min} = \min\nolimits_{e\in\E} a^e$ and constant $C'=2 a_{\min}^{-1} C e^T$ that is independent of $\eps$ and $t$. 
Together with Step~1 this completes the proof of the theorem.
\end{proof}

\subsection{Summary}

The previous theorem shows that the asymptotic analysis of convection--diffusion problems can be extended almost verbatim to networks, if appropriate coupling conditions and corresponding boundary layer functions are defined at the network junctions. By considering stationary problems or networks consisting only of a single pipe, one can see that the rate of the theorem can again not be improved. 

Before closing the presentation, let us mention some directions for further research: 
A natural next step would be to consider numerical approximations for singularly-perturbed convection-diffusion problems on networks. Based on the analysis given in this paper, we would expect that most of the results available for a single pipe, see \cite{RoosStynesTobiska} and the references given there, can be extended to networks. 
We would also expect that the convergence of the semigroup approach of \cite{Bardos70} can be extended to the network setting quite naturally.
Another point of interest might be to consider nonlinear problems and the asymptotic convergence in different metrics, which should be possible in the framework of entropy methods; 
we refer to \cite{Juengel16} for an introduction to the field.

\section*{Acknowledgements}
The authors are grateful for support by the German Research Foundation (DFG) via grants TRR~154 project C04 and TRR~146 project C03, 
and via the Center for Computational Engineering at TU Darmstadt.


\begin{thebibliography}{10}

\bibitem{Bardos70}
C.~Bardos.
\newblock Probl\`emes aux limites pour les \'{e}quations aux d\'{e}riv\'{e}es
  partielles du premier ordre \`a coefficients r\'{e}els; th\'{e}or\`emes
  d'approximation; application \`a l'\'{e}quation de transport.
\newblock {\em Ann. Sci. \'{E}cole Norm. Sup. (4)}, 3:185--233, 1970.

\bibitem{Bobisud67}
L.~Bobisud.
\newblock Second-order linear parabolic equations with a small parameter.
\newblock {\em Arch. Rational Mech. Anal.}, 27:385--397, 1967.

\bibitem{ChaconRebollo10}
T.~Chac\'{o}n~Rebollo, M.~G. M\'{a}rmol, and I.~S\'{a}nchez Mu\~{n}oz.
\newblock Analysis of a singular limit of boundary conditions for
  convection-diffusion equations.
\newblock {\em Asymptot. Anal.}, 70:141--154, 2010.

\bibitem{DL5}
R.~Dautray and J.-L. Lions.
\newblock {\em Mathematical analysis and numerical methods for science and
  technology. {V}ol. 5}.
\newblock Springer-Verlag, Berlin, 1992.
\newblock Evolution problems. I.

\bibitem{Dobrowolski97}
M.~Dobrowolski and H.-G. Roos.
\newblock A priori estimates for the solution of convection-diffusion problems
  and interpolation on {S}hishkin meshes.
\newblock {\em Z. Anal. Anwendungen}, 16:1001--1012, 1997.

\bibitem{Dorn08}
B.~Dorn.
\newblock Semigroups for flows in infinite networks.
\newblock {\em Semigroup Forum}, 76:341--356, 2008.

\bibitem{DornKramarNagelRadl10}
B.~Dorn, M.~Kramar~Fijav\v{z}, R.~Nagel, and A.~Radl.
\newblock The semigroup approach to transport processes in networks.
\newblock {\em Phys. D}, 239:1416--1421, 2010.

\bibitem{EggerKugler18}
H.~Egger and T.~Kugler.
\newblock Damped wave systems on networks: exponential stability and uniform
  approximations.
\newblock {\em Numer. Math.}, 138:839--867, 2018.

\bibitem{EggerPhilippi20}
H.~Egger and N.~Philippi.
\newblock A hybrid discontinuous {G}alerkin method for transport equations on
  networks.
\newblock {\em arXiv:2001.08004}, 2020.

\bibitem{EngelNagel}
K.-J. Engel and R.~Nagel.
\newblock {\em One-parameter semigroups for linear evolution equations}, volume
  194 of {\em Graduate Texts in Mathematics}.
\newblock Springer-Verlag, New York, 2000.
\newblock With contributions by S. Brendle, M. Campiti, T. Hahn, G. Metafune,
  G. Nickel, D. Pallara, C. Perazzoli, A. Rhandi, S. Romanelli and R.
  Schnaubelt.

\bibitem{Evans}
L.~C. Evans.
\newblock {\em Partial differential equations}, volume~19 of {\em Graduate
  Studies in Mathematics}.
\newblock American Mathematical Society, Providence, RI, second edition, 2010.

\bibitem{Garavello}
M.~Garavello and B.~Piccoli.
\newblock {\em Traffic flow on networks}, volume~1.
\newblock American institute of mathematical sciences Springfield, 2006.

\bibitem{Goering83}
H.~Goering, A.~Felgenhauer, G.~Lube, H.-G. Roos, and L.~Tobiska.
\newblock {\em Singularly perturbed differential equations}, volume~13.
\newblock Akademie-Verlag Berlin, 1983.

\bibitem{Juengel16}
A.~J\"ungel.
\newblock {\em Entropy methods for diffusive partial differential equations}.
\newblock Springer Briefs in Mathematics. Springer, 2016.

\bibitem{KramarSikolya05}
M.~Kramar and E.~Sikolya.
\newblock Spectral properties and asymptotic periodicity of flows in networks.
\newblock {\em Math. Z.}, 249:139--162, 2005.

\bibitem{Lagnese}
J.~E. Lagnese, G.~Leugering, and E.~G. Schmidt.
\newblock {\em Modeling, analysis and control of dynamic elastic multi-link
  structures}.
\newblock Springer Science \& Business Media, 2012.

\bibitem{Linss}
T.~Lin{\ss}.
\newblock {\em Layer-adapted meshes for reaction-convection-diffusion
  problems}.
\newblock Springer, 2009.

\bibitem{Mehmeti}
F.~Mehmeti, J.~Von~Below, and S.~Nicaise.
\newblock {\em Partial differential equations on multistructures}, volume 219.
\newblock CRC Press, 2001.

\bibitem{Miller12}
J.~J. Miller, E.~O'Riordan, and G.~I. Shishkin.
\newblock {\em Fitted numerical methods for singular perturbation problems:
  error estimates in the maximum norm for linear problems in one and two
  dimensions}.
\newblock World Scientific, 2012.

\bibitem{Mugnolo}
D.~Mugnolo.
\newblock {\em Semigroup methods for evolution equations on networks}.
\newblock Understanding Complex Systems. Springer, Cham, 2014.

\bibitem{Oppenheimer00}
S.~F. Oppenheimer.
\newblock A convection-diffusion problem in a network.
\newblock {\em Appl. Math. Comput.}, 112:223--240, 2000.

\bibitem{Pazy}
A.~Pazy.
\newblock {\em Semigroups of linear operators and applications to partial
  differential equations}, volume~44 of {\em Applied Mathematical Sciences}.
\newblock Springer-Verlag, New York, 1983.

\bibitem{RoosStynesTobiska}
H.-G. Roos, M.~Stynes, and L.~Tobiska.
\newblock {\em Robust numerical methods for singularly perturbed differential
  equations}, volume~24 of {\em Springer Series in Computational Mathematics}.
\newblock Springer-Verlag, Berlin, second edition, 2008.
\newblock Convection-diffusion-reaction and flow problems.

\bibitem{Shishkin}
G.~I. Shishkin and L.~P. Shishkina.
\newblock {\em Difference methods for singular perturbation problems}.
\newblock CRC Press, 2008.

\end{thebibliography}

\end{document}